\documentclass[reqno,12pt]{amsart}

\usepackage{epsf}
\usepackage{graphics}
\usepackage{amssymb}
\usepackage{amsmath}
\usepackage{enumitem}
\usepackage{mathrsfs}

\theoremstyle{plain}
\newtheorem*{theorem}{Theorem}

\newtheorem{lemma}{Lemma}
\newtheorem{proposition}{Proposition}

\theoremstyle{definition}
\newtheorem*{definition}{Definition}
\newtheorem{remark}{Remark}
\newtheorem*{question}{Question}

\def\N{\mathbb N}
\def\Z{\mathbb Z}
\def\C{\mathbb C}

\def\phi{\varphi}
\def\geq{\geqslant}
\def\leq{\leqslant}
\def\x{\times}
\def\lk{\mathop{lk}}
\DeclareMathOperator{\nul}{null}
\def\rk{\text{rank}\,}
\def\ker{\text{ker}}
\def\CC{\mathscr C}
\def\emptyset{\varnothing}

\title{Hopf bands in arborescent Hopf plumbings}

\author{Filip Misev}
\address{Universit\"at Bern, Sidlerstrasse 5, CH-3012 Bern, Switzerland}
\email{filip.misev@math.unibe.ch}
\date{}

\begin{document}

\begin{abstract} For a positive Hopf plumbed arborescent Seifert surface $S$, we study the set of Hopf bands $H\subset S$, up to homology and up to the action of the monodromy. The classification of Seifert surfaces for which this set is finite is closely related to the classification of finite Coxeter groups.
\end{abstract}

\maketitle
\thispagestyle{empty}

\section{Introduction}
Let $S\subset S^3$ be a Seifert surface of a link $L$, and let $q$ be the quadratic form on $H_1(S,\Z)$ associated with the Seifert form. Every oriented simple closed curve $\alpha\subset S$ can be thought of as a framed link in $S^3$, where the framing is induced by the surface $S$ and encoded by the value $q$ takes on the homology class represented by $\alpha$. For a fixed integer $n$, we are interested in the set $\CC_n(S)$ of isotopy classes of $n$-framed unknotted oriented curves $\alpha\subset S$. By Rudolph's work~\cite{Ru} on quasipositive surfaces we know for example $\CC_n(S)=\emptyset$ whenever $n\leq 0$ and $S$ is quasipositive. Here, we focus on positive arborescent (tree-like) Hopf plumbings, where $S$ is a surface obtained by an iterated plumbing of positive Hopf bands according to a finite plane tree $T$. Positive arborescent Hopf plumbings are particular examples of quasipositive surfaces. At the same time, they are fibre surfaces, i.e. pages of open books with binding $K=\partial S$. In fact every fibre surface in $S^3$ can be obtained from the standard disk by successively plumbing and deplumbing positive or negative Hopf bands. This results from Giroux' work on open books and contact structures (see the article~\cite{GG} by Giroux and Goodman). Not much is known about how (non-)unique a presentation of a given fibre surface $S$ as a plumbing of Hopf bands may be. In our previous article~\cite{Mi}, we have studied embedded arcs in fibre surfaces cutting along which corresponds to deplumbing a Hopf band, and we gave examples showing that the plumbing structure can be highly non-unique. Here, we take a similar, but different approach to understanding the plumbing structure of $S$ by studying the set $\CC_n(S)$ in the case $n=1$, whose elements correspond to Hopf bands that can potentially be deplumbed. The monodromy $\phi:S\to S$ of the open book acts on the set $\CC_n(S)$ as well as on its image $C_n(S)$ in $H_1(S,\Z)$, thus providing it with additional structure.
Finite trees can be divided into three families named {\em spherical}, {\em affine} and {\em hyperbolic}, according to the classification of Coxeter groups (compare~\cite{AC1,Hu}). The spherical trees comprise two families, called $A_n$ and $D_n$, plus three more trees named $E_6,E_7,E_8$, whereas the affine trees are denoted $\tilde{D_n},\tilde{E_6},\tilde{E_7},\tilde{E_8}$. Up to these exceptions, all trees fall into the class of hyperbolic trees.

\begin{theorem} Let $T$ be a finite plane tree and $S\subset S^3$ the corresponding positive arborescent Hopf plumbed surface. Then the set of homology classes of Hopf bands $C_1(S)$ is finite if and only if $T$ is spherical.\\In contrast, if $T$ is hyperbolic and $\partial S_T$ is a knot, $C_1(S)$ consists of infinitely many orbits of the monodromy.
\end{theorem}

Interestingly, the above correspondence between Coxeter groups and tree-like Hopf plumbings does not seem to be of purely homological nature: in the exceptional cases that correspond to affine Coxeter groups, there are in fact infinitely many $\phi$-orbits of homology classes $a\in H_1(S,\Z)$ such that $q(a)=1$. However, it is conceivable that only finitely many orbits can be realised by honest Hopf bands (that is, by unknotted embedded simple closed curves in $S$). We prove this for the smallest affine tree $\tilde{D}_4$, where it already suffices to exclude homology classes that are not representable by simple closed curves. In the case $\tilde{E_6}$, for example, infinitely many orbits can be realised by embedded (possibly knotted) simple closed curves. We did not succeed in finding unknotted representatives of these homology classes, though.

In the spherical cases, the set $C_1(S)$ coincides with $q^{-1}(1)\subset H_1(S,\Z)$ and consists of the ``obvious'' Hopf bands only, that is, simple combinations of the ones used in the plumbing construction, and their images under the monodromy. It might be interesting to study $\CC_n(S)$ or $C_n(S)$ for other $n$ and other classes of Seifert surfaces, as well as the case where unknotted curves are replaced by curves of a fixed knot type.

\subsection*{Plan of the article}
In the subsequent section we define and briefly discuss the concepts in question, that is, Seifert surfaces, the Seifert form, fibre surfaces, monodromy and Hopf plumbing. Section~\ref{sec:quadraticforms} concerns the integral quadratic forms that arise from trees. The above theorem is a consequence (in fact, a summary) of three propositions. Proposition~\ref{prop:spherical} concerns the spherical trees and is given in Section~\ref{sec:spherical}. In Section~\ref{sec:affine}, the affine trees are discussed in Proposition~\ref{prop:affine}. Finally we address the hyperbolic case with Propostion~\ref{prop:infiniteknots} in the last section.

\newpage

\section{Terms and definitions}

\subsection{Seifert surfaces, Seifert form}
Throughout, $S$ denotes a compact connected oriented surface with boundary embedded in the three-sphere, a {\em Seifert surface} for short. The {\em Seifert form} of a Seifert surface $S$ is a bilinear form $(.\,,.)$ on $H_1(S,\Z)$ defined on oriented simple closed curves $\alpha,\beta\subset S$ by $(\alpha,\beta)=\lk(\alpha,\beta^+)$, where $\beta^+$ is obtained by slightly pushing $\beta$ into $S^3\setminus S$ along the positive normal direction to $S$, and $\lk$ denotes the linking number. The Seifert form induces a quadratic form $q:H_1(S,\Z)\to\Z$ by $q(a)=-(a,a)$. For a simple closed curve $\alpha$, the integer $-q(\alpha)$ describes the framing\footnote{The sign makes sure that $q=+1$ on positive Hopf bands.} of an annular neighbourhood of $\alpha$ in $S$.

\subsection{Fibre surfaces, monodromy} \label{subsec:monodromy}
A Seifert surface $S$ is called a {\em fibre surface} if its interior $\mathring{S}$ is the fibre of a locally trivial fibre bundle $S^3\setminus\partial S\to S^1$, that is, $S^3\setminus\partial S$ has the structure of a mapping torus $(\mathring{S}\x [0,1])/_{(x,1)\sim (\phi(x),0)}$. The glueing homeomorphism $\phi:S\to S$, called the {\em monodromy}, is determined by the fibration up to isotopy. It is known that the Seifert matrix $V$ of a fibre surface $S$ with respect to a basis of $H_1(S,\Z)$ is invertible~\cite[Lemma~8.6]{BZH}. The matrix $M$ of the homological action of the monodromy with respect to the same basis can be computed using the formula $M=V^{-\top}V$, see~\cite[Lemma~8.3]{Sa}.

\subsection{Hopf plumbing}
Let $H$ be a Hopf band, that is, an unknotted annulus with a (positive or negative) full twist, as in Figure~\ref{fig:HopfPlumbing}. Let $S$ be a Seifert surface and let $I\subset S$ be a properly embedded interval with endpoints on $\partial S$ (an {\em arc} for short).
\begin{figure}[h]
\epsffile{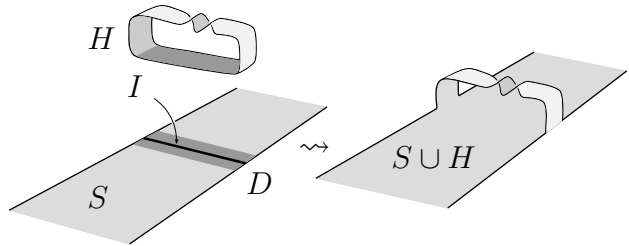}
\put(-205,75){$H$}
\put(-205,15){$S$}
\put(-190,57){$I$}
\put(-145,20){$D$}
\put(-125,35){$\leadsto$}
\put(-90,30){$S\cup H$}
\caption{Plumbing a Hopf band $H$ to a surface $S$ along $I$.\label{fig:HopfPlumbing}}
\end{figure}
Take a neighbourhood $D\subset S$ of $I$ such that $D$ is an embedded square with two opposite sides in $\partial S$. Thicken $D$ to one side of $S$ to obtain a box $B$ that intersects $S$ exactly in $D$. Similarly, take a square $D'\subset H$ with two opposite sides in $\partial H$. Place $H$ inside $B$, matching the square $D$ with $D'$ such that the sides $\partial H\cap D'$ are parallel to $I$. The surface $S\cup H$ is then said to be obtained from $S$ by {\em Hopf plumbing}. Two arcs that are isotopic in $S$ yield isotopic surfaces. Hopf plumbing preserves fibredness: If $S$ is a fibre surface, then so is any surface obtained from $S$ by Hopf plumbing, compare~\cite{St,Ga}.

\subsection{Positive arborescent Hopf plumbings}
Given a finite plane tree $T$, construct a fibre surface $S=S_T$ by taking one positive Hopf band for every vertex of $T$ and use plumbing to glue all pairs of Hopf bands that correspond to adjacent vertices in $T$, respecting the cyclic order of the edges adjacent to each vertex. A Seifert surface $S$ obtained in this way is called a {\em positive arborescent Hopf plumbing}. This construction is described and studied in greater generality by Bonahon and Siebenmann in their work on arborescent knots~\cite{BS}. For example, if $T$ is the tree $A_n$ shown in Figure~\ref{fig:spherical}, $S_T$ is the standard Seifert surface of the $(2,n+1)$ torus link. For $T=D_4$, we obtain the standard Seifert surface of the $(3,3)$ torus link. Yet another example is illustrated in Figure~\ref{fig:D5} below. The core curves (with a chosen orientation) of the Hopf bands used for the construction form a basis of $H_1(S,\Z)$. Relative to a basis, the Seifert quadratic form $q$ is a homogeneous polynomial of degree two in $r$ variables, where $r=\rk H_1(S,\Z)$ equals the number of vertices of $T$, or, equivalently, the number of Hopf bands used to construct $S_T$.

\begin{figure}[h]
\epsffile{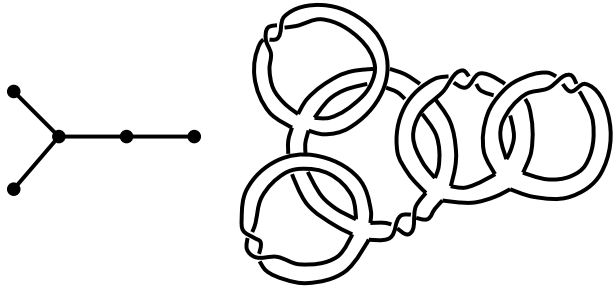}
\caption{The spherical tree $D_5$ and the corresponding Hopf plumbing $S_{D_5}$.\label{fig:D5}}
\end{figure}

\begin{remark}[on different notions of positivity]
There are different possible ways to define a Hopf band to be positive or negative. Here, by a {\em positive Hopf band} we mean an unknotted oriented band whose boundary link with the induced orientation is a {\em positive} Hopf link, so the core curve of a positive Hopf band has framing $-1$. Positive braid links are plumbings of positive Hopf bands. The subsequent statements can be translated into statements about negative Hopf plumbings.
\end{remark}

\section{Quadratic forms and Coxeter-Dynkin trees} \label{sec:quadraticforms}
Let $T$ be a finite tree with vertices $\{v_1,\ldots,v_n\}$. Use the same symbols $v_i$ to denote the basis of $H_1(S_T,\Z)\cong\Z^n$ consisting of the core curves of the plumbed Hopf bands. We define the matrix $A_T$ to be the symmetric integral $r\x r$-matrix whose diagonal entries are $2$'s and whose off-diagonal $ij$-th entry is $-1$ if $v_i$ and $v_j$ are connected by an edge in $T$ and $0$ otherwise. The orientations of the core curves may be chosen such that
\[ q_T(x_1v_1+\ldots+x_nv_n) = \frac12 x A_T x^\top,\quad \forall x=(x_1,\ldots,x_n)\in\Z^n. \]
Indeed, two non-adjacent Hopf bands being disjoint corresponds to the zero entries in $A_T$ and positive Hopf bands being $(-1)$-framed fits the diagonal entries. Finally, the core curves $v_i,v_j$ of two plumbed Hopf bands intersect exactly once and do not link otherwise. Thanks to the arborescent structure of the plumbing, we can choose orientations such that $\{\lk(v_i,v_j^+),\lk(v_j,v_i^+)\}=\{0,1\}$. In other words,
\[ -A_T=V+V^\top, \]
where $V$ is the Seifert matrix of $S_T$ with respect to the basis $v_1,\ldots,v_r$.

\noindent The quadratic form $q_T$ is
\begin{itemize}[leftmargin=*]
 \item positive definite if $T$ corresponds to a spherical Coxeter group,
 \item positive semidefinite if $T$ corresponds to an affine Coxeter group,
 \item indefinite otherwise.
\end{itemize}
Accordingly, we call a tree $T$ {\em spherical}, {\em affine}, {\em hyperbolic}. Compare Figure~\ref{fig:spherical} for a list of the spherical trees and Figure~\ref{fig:affine} for the affine trees. Any finite plane tree not appearing in one of these lists is hyperbolic. In fact, the so-called {\em slalom knots} introduced by A'Campo can be described as arborescent Hopf plumbings, and the slalom knots given by a hyperbolic tree are exactly the ones whose complements admit a complete hyperbolic metric of finite volume~\cite[Theorem~1]{AC2}.

As we already suggested in the introduction, there is a bijective correspondence between positive Hopf bands $H$ embedded in $S_T$ (up to isotopy in $S_T$) and $(-1)$-framed unknotted simple closed curves $\alpha\subset S_T$ (up to isotopy in $S_T$). The correspondence is given in one direction by assigning to a Hopf band $H$ its core curve $\alpha$, and in the other direction by setting $H$ to be a regular neighbourhood of $\alpha$ in $S_T$. Passing to homology classes, we can think of an element $x\in H_1(S_T,\Z)$ as the homology class of a positive Hopf band if and only if $x$ can be realised as the homology class of an unknotted simple closed curve in $S_T$ and $q_T(x)=1$.

\begin{definition}
We denote by $C_1(S_T)$ the set of homology classes of positive Hopf bands in $S_T$. Note that $C_1(S_T)\subset q_T^{-1}(1)\subset H_1(S_T,\Z)$.
\end{definition}

We complete this section with a few remarks that concern the above definition and which are important for the rest of the article.

\begin{remark} \label{rem:primitive}
If $\partial S_T$ is a knot, $x\in H_1(S_T,\Z)$ is representable by a simple closed curve if and only if it is {\em primitive}, i.e., if it cannot be written as a multiple of another vector (see~\cite[Proposition~6.2]{FM} for a proof in the closed case). In particular, any $x\in q_T^{-1}(1)$ can be realised by a (possibly knotted) simple closed curve if $\partial S_T$ is a knot.
\end{remark}

\begin{remark}
Let $T$ be a finite plane tree, $S_T$ the corresponding surface and $\phi:S_T\to S_T$ the monodromy. If $w\in H_1(S_T,\Z)$ is a homology class represented by an unknotted simple closed curve $\alpha$, then $\phi(\alpha)$ is again an unknotted simple closed curve, since the flow of the monodromy vector field describes an isotopy from $\alpha$ to $\phi(\alpha)$ in $S^3$. In addition, $\phi(\alpha)$ represents the homology class $\phi_*(w)$, whose framing equals the framing of $w$. In particular, $w\in C_1(S_T)$ iff $(\phi_*)^n(w)\in C_1(S_T)$, $\forall n\in\Z$.
\end{remark}

\begin{remark}
If $T'$ is a subtree of a tree $T$ (that is, $T'$ is obtained from $T$ by contracting edges), $S_{T'}$ can be viewed as a subsurface of $S_T$ in such a way that the map on homology induced by the inclusion is injective. In particular, $C_1(S_{T'})$ can be identified with a subset of $C_1(S_T)$.
\end{remark}

\begin{remark} \label{rem:standard}
We immediately spot a certain number of Hopf bands in an arborescent Hopf plumbing $S_T$, such as the Hopf bands corresponding to the vertices $v_i$ of $T$ that were used in the Hopf plumbing construction. Based on this observation we say that $x\in C_1(S_T)$ is a {\em standard Hopf band} if $x\in C_1(S_{T'})$ for some subtree $T'\subset T$ of type $A_n$. See Figure~\ref{fig:A5} for an example and the paragraph after Proposition~\ref{prop:spherical} for a complete description of all standard Hopf bands in $S_{A_n}$.
\end{remark}

\begin{figure}[h]
\epsffile{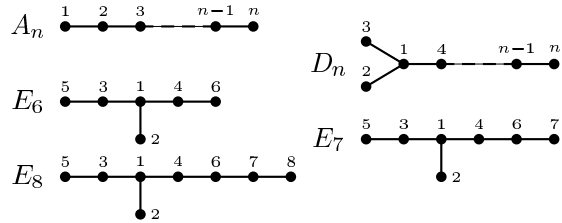}
\caption{The simply laced spherical Coxeter trees correspond to finite Coxeter groups. The numbers indicate the order of the chosen homology basis vectors.\label{fig:spherical}}
\end{figure}



\section{The spherical Coxeter-Dynkin trees} \label{sec:spherical}
If $T$ is one of the spherical trees $A_n,D_n,E_6,E_7$ or $E_8$ depicted in Figure~\ref{fig:spherical}, the quadratic form $q_T$ is positive definite. Therefore, the equation $q_T(x)=k$ has only finitely many integral solutions, for any fixed $k$ and in particular for $k=1$. In the rest of this section, we explicitely determine all solutions to $q_T(x)=1$ for each of the spherical trees. These were already studied and classified in the context of Lie algebra theory and Coxeter groups, see for example the book by Humphreys~\cite{Hu}. The following proposition summarises the results of this section (see Remark~\ref{rem:standard} for the definition of a standard Hopf band).

\begin{proposition} \label{prop:spherical}
If $T$ is a spherical tree, then the set of integral solutions to $q_T(x)=1$ is finite. Moreover, every solution is contained in the orbit of a standard Hopf band under the monodromy and is therefore realisable as an unknotted simple closed curve in $S_T$. In particular, $C_1(S_T)=q_T^{-1}(1)$.
\end{proposition}

First let $T=A_n$. The associated quadratic form $q$ then takes the following form with respect to the basis of $H_1(S_T,\Z)\cong\Z^n$ described above:
\begin{eqnarray}
 q(x) &=& x_1^2+\ldots+x_n^2 \ - x_1x_2 - x_2x_3 - \ldots - x_{n-1}x_n \nonumber \\
      &=& \frac12 ((x_1-x_2)^2 + \ldots + (x_{n-1}-x_n)^2 + x_1^2 + x_n^2) \nonumber
\end{eqnarray}
For any integral solution $x$ to $q(x)=1$, necessarily $|x_1|,|x_n|\leq 1$. Therefore, we obtain
\[ C_1(S_{A_n}) = \{ \pm({\bf 0}_r,{\bf 1}_s,{\bf 0}_t)\in\Z^n\ | \ r,t\geq 0,s\geq 1 \}, \]
where ${\bf c}_\nu$ stands for $\nu$ consecutive occurences of the number $c$. It is easily seen that all $n(n+1)$ elements of the above set can be represented by unknotted simple closed curves in $S_{A_n}$, see Figure~\ref{fig:A5} for an example.

\begin{figure}[h]
\epsffile{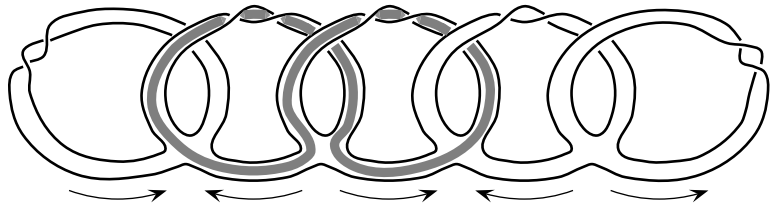}
\caption{A standard Hopf band (in grey) in $S_{A_5}$, representing the homology class $(0,1,1,0,0)$ with respect to the homology basis described above. The arrows indicate the chosen orientations of the homology basis vectors.\label{fig:A5}}
\end{figure}

\noindent For $T=D_n$, one finds:
\begin{eqnarray}
 q(x) &=& x_1^2+\ldots+x_n^2 \ - x_1(x_2+x_3+x_4) - x_4x_5 - \ldots - x_{n-1}x_n \nonumber \\
      &=& (\frac12 x_1-x_2)^2 + (\frac12 x_1-x_3)^2 + \frac12(x_1-x_4)^2 \nonumber \\
      &&  +\frac12((x_4-x_5)^2 + \ldots + (x_{n-1}-x_n)^2) + \frac12 x_n^2 \nonumber
\end{eqnarray}

Let $x\in\Z^n$ be a solution of $q(x)=1$. Then $|x_n|\leq 1$, since otherwise the last summand would already be larger than one. If $x_n=0$, then $x\in C_1(S_{D_{n-1}})$. Otherwise, we may assume $x_n=1$, up to changing the sign of $x$. If $x_1=0$, then $x_2,x_3$ cannot be both nonzero, so we can view $x\in C_1(S_{A_{n-1}})$. Hence we may assume $x_1\neq 0$. If $x_1$ is odd, then the first two squares are in $\frac14\N\setminus\N$ while the last summand is equal to $\frac12$. This implies that the rest vanishes, i.e., $x_1=x_4=x_5=\ldots = x_n = 1$, and hence $x_2,x_3\in\{0,1\}$. If $x_1$ is even, the first two squares are in $\N$ and must therefore vanish, while
\[ (x_1-x_4)^2+(x_4-x_5)^2+\ldots + (x_{n-2}-x_{n-1})^2+(x_{n-1}-1)^2 = 1. \]
This implies that exactly one of these squares equals one and the rest vanishes, so we obtain the following solutions (recall that $x_1\neq 0$ by assumption).
\[ x=(2,1,1,{\bf 2}_r,{\bf 1}_s)\in\Z^n,\quad r\geq 0,\ s\geq 1 \]
If $M$ denotes the matrix of the homological action of the monodromy with respect to the basis $v_1,\ldots,v_n$, the following relations hold for $r\geq 0$, $s\geq 1$.
\begin{eqnarray}
 ({\bf 1}_{r+3},{\bf 0}_{n-r-3})^\top               &=& (-1)^{r+1} M^{r+2}({\bf 0}_{n-r-1},{\bf 1}_{r+1})^\top \nonumber \\
 (2,1,1,{\bf 2}_r,{\bf 1}_s,{\bf 0}_{n-r-s-3})^\top &=& (-1)^{r+1} M^{r+1}(1,0,0,{\bf 1}_{s-1},{\bf 0}_{n-s-2})^\top \nonumber
\end{eqnarray}
Hence all solutions to $q(x)=1$ are contained in orbits of standard Hopf bands under the monodromy and are therefore realisable by unknotted simple closed curves in $S_T$. In total, we obtain
\begin{eqnarray}
  C_1(S_{D_n})\! &=& \! \{ \pm (2,1,1,{\bf 2}_r,{\bf 1}_s,{\bf 0}_t)\in\Z^n \ | \ r,t\geq 0,\ s\geq 1\} \nonumber \\
                 && \cup\ \{ \pm (1,x_2,x_3,{\bf 1}_r,{\bf 0}_s)\in\Z^n \ | \ x_2,x_3\in\{0,1\},\ r,s\geq 0 \} \nonumber \\
                 && \cup\ \{ \pm ({\bf 0}_r,{\bf 1}_s,{\bf 0}_t)\in\Z^n \ | \ r\geq 3,\ s\geq 1\}\ \cup \ \{ \pm v_2, \pm v_3 \}. \nonumber 
\end{eqnarray}
A combinatorial calculation shows that $\#C_1(S_{D_n})=2n(n-1)$.

\noindent For $T=E_6$, the quadratic form $q$ takes the following form.
\begin{eqnarray}
  q(x)\!\!\! &=& \!\!\!x_1^2+\ldots+x_6^2 \ - x_1(x_2+x_3+x_4) - x_3x_5 - x_4x_6 \nonumber \\
             &=& \!\!\! (\frac12 x_1 - x_2)^2  + \frac13 ((x_1 - \frac32 x_3)^2  +  (x_1 - \frac32 x_4)^2) \nonumber \\
             && \!\!\! + (\frac12 x_3 - x_5)^2  +  (\frac12x_4 - x_6)^2  +  \frac{1}{12} x_1^2 \nonumber
\end{eqnarray}
$E_6$ contains $A_5$ and $D_5$ as subtrees, whose sets of Hopf bands we know. Let $x$ be a solution of $q(x)=1$ different from these (in particular, $x_2,x_5,x_6\neq 0$). From the condition $\frac{1}{12}x_1^2\leq 1$ we obtain $|x_1|\leq 3$. If $x_1=0$, at most one of $x_2,x_3,x_4$ can be nonzero. It follows that $x$ is supported in one of the arms of the tree, which we excluded. Therefore we may assume $x_1\in\{1,2,3\}$ (up to changing the sign of $x$). If $x_1=3$, then $\frac{1}{12}x_1^2=\frac34$ and $(\frac12 x_1-x_2)^2\in\frac14\N\setminus\N$. So $x_2\in\{1,2\}$ and $x_3=x_4=2$, $x_5=x_6=1$, hence
\[ x=(3,x_2,2,2,1,1), \quad x_2\in\{1,2\}. \]
If $x_1=2$, then $x_2=1$ (otherwise $(\frac12x_1-x_2)^2+\frac{1}{12}x_1^2>1$), and similarly $x_3,x_4\in\{1,2\}$, which implies $x_5=x_6=1$ (remember we excluded $x_5=0$ or $x_6=0$). This yields the four possibilities
\[ x=(2,1,x_3,x_4,1,1),\quad x_3,x_4\in\{1,2\}. \]
Finally, if $x_1=1$, we can successively deduce $|x_i|\leq 1$ for $i=2,\ldots,6$, hence
\[ x=(1,1,1,1,1,1). \]
It turns out that all of these homology classes can be realised by unknotted simple closed curves. In fact, they all lie, up to sign, in the orbits of $v_1, v_1+v_2,v_1+v_3,v_1+v_4$ under the monodromy. In total, we have $14$ vectors plus $2\cdot(\#C_1(S_{D_5}))-\#C_1(S_{D_4})+2=58$ vectors coming from the subtrees $D_5$ and $A_5$, so $\# C_1(S_{E_6})=72$.

\noindent For $E_7$, we obtain
\begin{eqnarray}
  q(x)\!\!\! &=& \!\!\!x_1^2+\ldots+x_7^2 \ - x_1(x_2+x_3+x_4) - x_3x_5 - x_4x_6 - x_6x_7 \nonumber \\
             &=& \!\!\! (\frac12 x_1 - x_2)^2  + \frac13 ((x_1 - \frac32 x_3)^2  +  (x_4 - \frac32 x_6)^2) \nonumber \\
             && \!\!\! +  \frac23 (\frac34 x_1 - x_4)^2  +  (\frac12x_3 - x_5)^2  +  (\frac12 x_6 - x_7)^2  +  \frac{1}{24} x_1^2 \nonumber
\end{eqnarray}
The vectors $x$ satisfying $q(x)=1$ that are not supported on the $D_6$ or $E_6$ subtrees are (up to sign)
\[ \begin{array}{lll}
    (4,2,3,3,x_5,2,1),&(3,x_2,2,x_4,1,2,1),&(2,1,x_3,2,1,x_6,1),\\
    (4,2,2,3,1,2,1),  &(3,x_2,2,2,1,1,1),  &(2,1,x_3,1,1,1,1), \\
    (1,1,1,1,1,1,1),&&
   \end{array} \]
where $x_2,x_3,x_5,x_6\in\{1,2\}$ and $x_4\in\{2,3\}$ can be freely chosen. As before, these homology classes are contained in the orbits of the Hopf bands $v_1$, $v_2$, $v_3$, $v_1+v_3$, $v_1+v_4$, $v_4+v_6$, $v_1+v_4+v_6$ under the monodromy, hence realisable by unknotted simple closed curves. A count of elements yields $\#C_1(S_{E_7})=126$.

\noindent Finally, for $E_8$
\begin{eqnarray}
  q(x)\!\!\! &=& \!\!\!x_1^2+\ldots+x_8^2 \ - x_1(x_2+x_3+x_4) - x_3x_5 - x_4x_6 - x_6x_7 - x_7x_8 \nonumber \\
             &=& \!\!\! (\frac12 x_1 - x_2)^2  + \frac13 ((x_1 - \frac32 x_3)^2  +  (x_6 - \frac32 x_7)^2) +  \frac25 (x_1 - \frac54x_4)^2 \nonumber \\
             && \!\!\! + \frac38(x_4 - \frac43x_6)^2  +  (\frac12 x_3 - x_5)^2  + (\frac12x_7 - x_8)^2 +  \frac{1}{60} x_1^2 \nonumber
\end{eqnarray}
The vectors not supported on the $D_7$ or $E_7$ subtrees are
\[ \begin{array}{lll}
    (6,3,4,5,2,4,3,x_8),&(6,3,4,5,2,x_5,2,1),&(6,3,4,4,2,3,2,1),\\
    (5,x_2,4,4,2,3,2,1),&(5,x_2,3,4,y_5,3,2,1),&(4,2,3,4,y_5,3,2,1),\\
    (4,2,2,4,1,3,2,1),&(4,2,3,3,y_5,3,2,1),&(4,2,3,3,y_5,2,x_7,1),\\
    (4,2,2,3,1,x_6,2,1),&(4,2,2,3,1,2,1,1),&(3,y_2,2,3,1,3,2,1),\\
    (3,y_2,2,3,1,2,x_7,1),&(3,y_2,2,2,1,2,x_7,1),&(3,y_2,2,2,1,1,1,1),\\
    (2,1,x_3,2,1,2,x_7,1),&(2,1,x_3,x_4,1,1,1,1),&(1,1,1,1,1,1,1,1),
\end{array} \]
where $x_3,x_4,x_7,x_8,y_2,y_5\in\{1,2\}$, $x_2,x_6\in\{2,3\}$, $x_5\in\{3,4\}$. Again, these vectors all lie in the orbits of $v_1$, $v_2$, $v_3$, $v_4$, $v_1+v_2$, $v_1+v_3$, $v_1+v_4$, $v_6+v_7$, $v_7+v_8$, $v_1+v_4+v_6$ under the monodromy and are therefore represented by unknotted simple closed curves. Together with the vectors supported on the $D_7$ and $E_7$ subtrees, they add up to a total count of $\# C_1(S_{E_8})=240$.

\section{The affine Coxeter-Dynkin trees} \label{sec:affine}
\begin{proposition} \label{prop:affine}
If $T$ is an affine tree, the set of solutions to $q_T(x)=1$ contains at least one infinite orbit of a standard Hopf band under the monodromy. In particular, the set $C_1(S_T)$ is infinite. More precisely, there exist vectors $u,w_1,\ldots,w_d\in H_1(S_T,\Z)$, such that every solution $x$ to $q_T(x)=1$ is of the form $x=w_i+ku$ for some $k\in\Z$, $i\in\{1,\ldots,d\}$.
\end{proposition}
\begin{figure}[h]
\epsffile{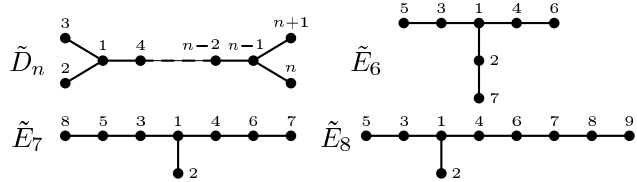}
\caption{The simply laced affine Coxeter trees.\label{fig:affine}}
\end{figure}
Let $T$ be an affine tree and let $n$ be the number of vertices of $T$. Observe that $T$ is obtained by adding one edge and one vertex $v_n$ to a suitable spherical subtree $T'$. In terms of the associated quadratic forms $q_T$, $q_{T'}$ this means: there exists a subspace of codimension one on which $q_T$ has a positive definite restriction, i.e., the radical $\ker A_T$ of $q_T(x)=\frac12 x A_T x^\top$ is one-dimensional. Let $u\in\Z^n$ be such that $\Z u=\ker A_T$. Observe that
\[ q_T(w+ku)=q_T(w)+k\, w A_T u^\top+k^2q_T(u)=q_T(w), \]
for all $w\in\Z^n$, $k\in\Z$. Let $V$ be the Seifert matrix of $S_T$ with respect to the basis $v_1,\ldots,v_n$ of $H_1(S_T,\Z)$ and denote by $M$ the matrix of the monodromy of $S_T$ with respect to the same basis. As mentioned in Section~\ref{subsec:monodromy} and in Section~\ref{sec:quadraticforms} respectively, the Seifert matrix $V$ is invertible and the following relations hold.
\[ M=V^{-\top}V,\quad -A_T=V+V^\top \]
Since $u\in\ker A_T$, we have $Vu=-V^\top u$ and therefore $Mu=-u$. Moreover, we will see that the last coordinate of $u$ (corresponding to the vertex $v_n$ of $T$) equals $\pm 1$. Therefore, $k\in\Z$ can be chosen such that $w+ku$ is supported in $T'$, so
\[ q_T^{-1}(1)=\{ w+ku\ | \ w\in q_{T'}^{-1}(1)\x\{0\},\ k\in\Z \}. \]
To simplify notation, we identify the vector $(x_1,\ldots,x_{n-1})\in\Z^{n-1}$ with the vector $(x_1,\ldots,x_{n-1},0)\in\Z^n$. It should be clear from the context when a last entry equal to zero is to be deleted from a vector and when a zero should be appended to a vector.

The smallest example of an affine tree $T$ is the the ``X''-shaped tree $\tilde{D_4}$ with five vertices $v_1,\ldots,v_5$ where $v_1$ has degree four and $v_2,\ldots,v_5$ have degree one (compare Figure~\ref{fig:affine}). In that case, $q(x)=q_T(x)$ can be written as a sum of four squares:
\begin{eqnarray} 
  q(x) &=& x_1^2+x_2^2+x_3^2+x_4^2+x_5^2\ - x_1(x_2+x_3+x_4+x_5) \nonumber \\
       &=& (\frac12 x_1-x_2)^2 + (\frac12 x_1-x_3)^2 + (\frac12 x_1-x_4)^2 + (\frac12 x_1-x_5)^2 \nonumber
\end{eqnarray}
One easily finds that the vector $u=(2,1,1,1,1)$ spans the radical. The subtree whose vertices are $v_1,\ldots,v_4$ is of type $D_4$. Therefore, all solutions $x$ to $q(x)=1$ are of the form $x=w+ku$, where $k\in\Z$ and $w\in C_1(S_{D_4})\subset\Z^ 4$. We know from the previous section on spherical trees that $C_1(S_{D_4})$ consists of the following vectors, up to sign and up to permutation of the last three entries:
\begin{eqnarray}
 w_1 = (2,1,1,1), && w_2 = (1,1,1,1),\quad w_3 = (1,1,1,0), \nonumber \\
 w_4 = (1,1,0,0), && w_5 = (1,0,0,0),\quad w_6 = (0,1,0,0). \nonumber
\end{eqnarray}
The matrix $M$ of the (homological) monodromy with respect to the basis $v_1,\ldots,v_5$ of $H_1(S_T,\Z)$ can be computed using the formula in terms of the Seifert matrix $V$ mentioned above. Concretely:
\[ M=\left( \begin{array}{rrrrr} -3&1&1&1&1\\ -1&1&&& \\ -1&&1&& \\ -1&&&1& \\ -1&&&&1 \end{array} \right) \]
The following relations hold:
\[ \begin{array}{lll}
 M^2(w_1) = w_1+u, & \ M^2(w_2) = w_2-u,  & \framebox{$M^2(w_3)=w_3$} \\[1ex]
 M^2(w_4) = w_4+u, & \ M^2(w_5) = w_5+2u, & \ M^2(w_6)=w_6-u
\end{array} \]
We will now deduce from these relations that the set of homology classes of Hopf bands $C_1(S_T)$ decomposes into finitely many orbits under the action of the monodromy $M$. Indeed, the above relations imply that the families $\{w_i+ku\}_{k\in\Z}$, $i\neq 3$, fall into finitely many orbits. On the other hand, each of the vectors $w^{(k)}:=w_3+ku$, $k\in\Z$, is fixed by $M^2$, so they cannot be contained in a finite union of $M$-orbits. However, $w^{(k)}\notin C_1(S_T)$ for $k\neq 0,-1$, since these homology classes cannot be represented by a simple closed curve in $S_T$, as we demonstrate now. $S_T$ is a surface of genus one with four boundary components. If we forget about the embedding of $S_T\subset S^3$, we can abstractly glue three disks to cap off all but one boundary component. The result is an abstract (non-embedded) once punctured torus. Therein, a nonzero first homology class $c$ can be represented by a simple closed curve if and only if $c$ is a primitive vector, that is, if $c=\lambda c'$ implies $|\lambda|=1$ for $\lambda\in\Z$ (compare Remark~\ref{rem:primitive}). To make use of this criterion, we change the basis $(v_1,\ldots,v_5)$ of $H_1(S_T,\Z)$ to the basis $(v_1,v_2,v_2-v_3,v_3-v_4,v_4-v_5)$. The last three elements of this new basis can be represented by three of the four boundary curves of $S_T$. Therefore, capping off these three boundary components has the effect of deleting the last three entries of the corresponding coordinate vectors. Rewriting $w^{(k)}$ in the new coordinates yields the vector
\[ (2k+1,4k+2,-(3k+1),-2k,-k). \]
Under the inclusion of $S_T$ into the capped-off surface, we obtain the vector
\[ (2k+1,4k+2)=(2k+1)\cdot(1,2), \]
which is primitive for $k\in\{0,-1\}$ only. Hence $w^{(k)}\notin C_1(S_T)$ for $k\neq 0,-1$ and the set $C_1(S_T)$ decomposes into finitely many orbits under $M$.

\noindent For the other members of the $\tilde{D_n}$ family ($n\geq 5$), we obtain
\begin{eqnarray}
 q(x) &=& (\frac12 x_1-x_2)^2+(\frac12 x_1-x_3)^2 \nonumber \\
      & & +\frac12\left( (x_1-x_4)^2+(x_4-x_5)^2+\ldots +(x_{n-2}-x_{n-1})^2\right) \nonumber \\
      & & +(\frac12 x_{n-1}-x_n)^2+(\frac12 x_{n-1}-x_{n+1})^2 \nonumber
\end{eqnarray}
whose radical is spanned by $u=(2,1,1,{\bf 2}_{n-4},1,1)$. As for $\tilde{D_4}$, the solutions to $q(x)=1$ are the vectors of the form $w_i+ku$, where $k\in\Z$ and $w_i\in C_1(S_{D_n})$, $i\in\{1,\ldots,2n(n-1)\}$, are the homology classes of Hopf bands supported in the $D_n$ subtree of $\tilde{D_n}$. A calculation with the monodromy matrix $M$ and the first standard Hopf band $v_1$ shows that $M^{n-2}v_1=(-1)^n (v_1+2u)$. Since $Mu=-u$, this implies
\[ M^{(n-2)k}v_1=(-1)^{nk} (v_1+2ku),\quad \forall k\in\Z. \]
Therefore, there is at least one infinite orbit of homology classes of Hopf bands in $C_1(S_T)$. In contrast, there do exist $w_i$ such that the family $\{w_i+ku\}_{k\in\Z}$ does not fall into finitely many orbits under the monodromy and still consists of homology classes of simple closed curves. For example, the vectors $w=({\bf 0}_r,{\bf 1}_s,{\bf 0}_t)\in\Z^{n+1}$, $3\leq r\leq n-1$, $1\leq s\leq n-r-1$, are in fact all fixed by $M^{2(n-2)}$, and $w+ku$ is realisable as a simple closed curve, for every $k\in\Z$. However, we do not know whether it can be realised as an {\em unknotted} simple closed curve for $k\notin\{0,-1\}$. The same situation occurs for $\tilde{E_6}$, $\tilde{E_7}$ and $\tilde{E_8}$. For $T=\tilde{E_6}$, the corresponding quadratic form $q$ can be written as follows. 
\begin{eqnarray}
  q(x) &=& \frac13 \left( (x_1-\frac32 x_2)^2+(x_1-\frac32 x_3)^2+(x_1-\frac32 x_4)^2\right) \nonumber \\
       & & + (\frac12 x_2-x_7)^2+(\frac12 x_3-x_5)^2+(\frac12 x_4-x_6)^2 \nonumber 
\end{eqnarray}
The radical of $q$ is spanned by the vector $u=(3,2,2,2,1,1,1)$, and $M^2v_1=v_1+u$, where $M$ and $v_1$ denote again the monodromy and the first standard Hopf band, respectively. This implies that the orbit of $v_1$ under the monodromy is infinite.
For the tree $\tilde{E}_7$, we have:
\begin{eqnarray}
  q(x) &=& (\frac12 x_1-x_2)^2 + (\frac12 x_5-x_8)^2+(\frac12 x_6-x_7)^2 \nonumber \\
       & & + \frac23 \left( (\frac34 x_1-x_3)^2+(\frac34 x_1-x_4)^2\right) \nonumber \\
       & & + \frac13 \left( (x_3-\frac32 x_5)^2+(x_4-\frac32 x_6)^2\right) \nonumber
\end{eqnarray}
The radical is generated by $u=(4,2,3,3,2,2,1,1)$ and one verifies the relation $M^3 v_1 = -v_1 -u$.
Finally, for $\tilde{E}_8$, we obtain:
\begin{eqnarray}
  q(x) &=& (\frac12 x_1-x_2)^2 + (\frac12 x_3-x_5)^2+(\frac12 x_8-x_9)^2  \nonumber \\
       & & + \frac35 (\frac56 x_1-x_4)^2 + \frac25 (x_4-\frac54 x_6)^2 + \frac23 (\frac34 x_6-x_7)^2 \nonumber \\
       & & + \frac13 \left( (x_1-\frac32 x_3)^2+(x_7-\frac32 x_8)^2\right) \nonumber
\end{eqnarray}
The radical is the span of $u=(6,3,4,5,2,4,3,2,1)$, $M^5v_1=-v_1-u$.

In summary, every solution $x$ to the equation $q_T(x)=1$ ($T$ affine) is of the form $x=w_i+ku$ for some $k\in\Z$, where $u$ generates the radical of $q_T$ and the $w_i$ are finitely many homology classes of Hopf bands contained in $S_{T'}$ for a spherical subtree $T'\subset T$. For certain $i$, the family $\{w_i+ku\}_{k\in\Z}$ is contained in finitely many orbits under the monodromy $M$ of $S_T$, while the members of the remaining families are fixed by some power $M^d$. Among the latter, there are homology classes that cannot be realised by simple closed curves, and there are such families whose members are realisable by simple closed curves.
\begin{question}
Can these homology classes be realised by {\em unknotted} simple closed curves? In other words, does $C_1(S_T)$ decompose into finitely many orbits under the monodromy, for any affine tree $T$?
\end{question}

\section{Infinite sets of orbits for hyperbolic trees} \label{sec:hyperbolic}
As described above, $C_1(S_T)$ is finite for the spherical trees and infinite for the affine trees. However, $C_1(S_T)$ could still decompose into finitely many orbits under the monodromy. We claim this is not anymore true for hyperbolic trees, at least when $\partial S_T$ is a knot.

\begin{proposition} \label{prop:infiniteknots}
Let $T$ be a hyperbolic tree, let $S_T$ be the corresponding fibre surface and denote the monodromy by $\phi$. If $\partial S_T$ is a knot, then the set $C_1(S_T)$ of homology classes of Hopf bands consists of infinitely many $\phi_*$-orbits.
\end{proposition}

The key idea for proving the proposition is to take an affine subtree $T'\subset T$ and to compare the action of the monodromy $\phi'_*$ on $H_1(S_{T'},\Z)$ with the action of $\phi_*$ on $H_1(S_T,\Z)$. In the previous section, we found $\phi'_*$-orbits of Hopf bands consisting of vectors $v_k\in H_1(S_{T'},\Z)$ that grow linearly in $k\in\Z$ with respect to any norm. This was possible because the Jordan normal form of $\phi'_*$ has Jordan blocks of size two to eigenvalues which are roots of unity. However, it turns out that the monodromy $\phi_*$ of the larger surface $S_T$ does not have such Jordan blocks. Therefore, the family $v_k$ cannot be a union of finitely many orbits under $\phi_*$ since the ``gaps" between consecutive members of an orbit must either stay bounded or grow exponentially. More specifically, the proposition will follow from the two subsequent lemmas.

\begin{lemma} \label{lem:jordanblocks}
Let $T$ be a hyperbolic tree such that $\partial S_T$ is a knot, and denote the corresponding monodromy by $\phi:S_T\to S_T$. The Jordan normal form of $\phi_*:H_1(S_T,\Z)\to H_1(S_T,\Z)$ cannot contain any Jordan block of size greater than one to an eigenvalue of modulus one.
\end{lemma}

\begin{proof}
The main ideas for the proof are contained in an appendix by Feller and Liechti to an article of Liechti~\cite{Li1}, see also~\cite{GL}. However, we choose to reformulate them here for the reader's convenience. Let $A$ be a Seifert matrix for $S_T$ with respect to some basis of $H_1(S_T,\Z)$. Then, the monodromy $\phi_*$ has matrix $A^{-\top}A$ with respect to that basis. Given an eigenvalue $\omega$ of $\phi_*$, we denote the algebraic and geometric multiplicities of $\omega$ by $m_{alg}(\omega)$ and $m_{geom}(\omega)$, respectively. Thus, $m_{alg}(\omega)$ is the multiplicity of the zero $\omega$ of the characteristic polynomial of $\phi_*$, while $m_{geom}(\omega)$ is the number of Jordan blocks to $\omega$ in the Jordan normal form of $\phi_*$. Our goal is to prove $m_{geom}(\omega)=m_{alg}(\omega)$ for every eigenvalue $\omega\in S^1$ of $\phi_*$, or, equivalently, of $\phi_*^{-1}$. Suppose to this end that $\omega\in S^1$ is an eigenvalue of $\phi_*^{-1}=A^{-1}A^\top$. Then we have
\[ 0=\det(A)\det(\omega I - \phi_*^{-1})=\det(\omega A - A^\top)=\Delta_K(\omega), \]
where $\Delta_K(t)$ denotes the Alexander polynomial of the knot $K=\partial S_T$ and $I$ denotes the identity map on $H_1(S_T,\Z)$. For $\omega\in S^1$, the $\omega$-signature (after Levine and Tristram~\cite{Tr}) is defined to be the signature $\sigma_\omega\in\Z$ of the Hermitian matrix
\[ M_\omega = (1-\omega)A + (1-\bar{\omega})A^\top = -(1-\bar{\omega})(\omega A - A^\top), \]
that is, the number of positive eigenvalues minus the number of negative eigenvalues of $M_\omega$. As $\omega=e^{\pi i t}$ traverses one half of the unit circle for $t\in(0,1]$, the $\omega$-signature $\sigma_\omega$ stays constant except at zeros $\omega_0$ of $\Delta_K$ (by the first equation above), where it may jump by some even amount $2j_{\omega_0}$. In addition, if $\omega\in S^1$ is an eigenvalue of $\phi_*^{-1}$,
\[ |j_\omega | \leq \nul(M_\omega) = \nul(\omega I- \phi_*^{-1}) = m_{geom}(\omega) \leq m_{alg}(\omega), \]
where $\nul(B)$ denotes the nullity of a matrix $B$, that is, the geometric multiplicity of the eigenvalue $0$ of $B$. For $\omega\in S^1$ near $1$ we have $\sigma_\omega=0$, whereas for $\omega=-1$, $\sigma_\omega$ takes the value of the classical signature invariant for knots, $\sigma(K)$. This implies
\[ \sigma(K) = \sigma_{-1}(K) \leq 2\cdot\!\sum_{\omega\in S_+^1}|j_\omega| \leq 2\cdot\!\sum_{\omega\in S_+^1} m_{alg}(\omega) = \sigma(K), \]
where $S_+^1=\{e^{\pi i t}\ | \ t\in (0,1]\}$ denotes the upper half of the unit circle. The last equality follows from the fact that the number of zeros of $\Delta_L$ on $S^1$ (counted with multiplicity) equals $\sigma(L)+\nul(L)$ for any tree-like Hopf plumbing $L$, where $\nul(L)=\nul(M_{-1})$ equals zero if $L$ is a knot. This is proven in a preprint by Liechti~\cite{Li2}. Therefore the above inequalities are in fact equalities, which implies $m_{geom}(\omega) = m_{alg}(\omega)$ for all zeros $\omega\in S_+^1$ of $\Delta_K$. By the symmetry of Alexander polynomials, the same holds for the zeros $\omega$ of $\Delta_K(t)$ such that $-\omega\in S_+^1$. Finally, $\Delta_K(1)\neq 0$ because $K$ is a knot. This ends the proof.
\end{proof}

\begin{lemma} \label{lem:exponential}
Let $\Phi$ be a matrix of size $n\x n$ with coefficients in $\C$, and let $v\in\C^n$ be any vector. Suppose that the Jordan normal form of $\Phi$ does not contain any Jordan block of size greater than one to an eigenvalue of modulus one. Then the sequence $\{\Phi^k(v)\}_{k\in\N}$ is either bounded or grows exponentially (with respect to any norm $\|.\|$ on $\C^n$).
\end{lemma}

\noindent By exponential growth of a sequence $\{v_k\}_{k\in\N}\subset \C^n$ with respect to a norm $\|.\|$ we mean the existence of constants $h>1$, $c>0$, $d\geq 0$, such that $\|v_k\|\geq ch^k-d$ for all $k\in\N$.

\begin{proof}[Proof of Lemma~\ref{lem:exponential}]
We may first assume that $\Phi$ is already in Jordan normal form, and second that $\Phi$ consists of just one Jordan block to some eigenvalue $\lambda$. If $v=0$, the conclusion is clear, so we assume $v\neq 0$. If $|\lambda|<1$, the sequence $\Phi^k(v)$ is bounded, if $|\lambda|>1$, it grows exponentially and if $|\lambda|=1$, $\Phi$ is of size one, so $\Phi^k(v)$ is bounded.
\end{proof}

\begin{proof}[Proof of Proposition~\ref{prop:infiniteknots}]
Let $T$ be a hyperbolic tree. Since $T$ contains an affine subtree $T'$, we have at least one infinite family $\{v_k\}_{k\in\N}$ of elements $v_k\in C_1(S_{T'})\subset C_1(S_T)$ which grow linearly in $k$ when seen as a sequence of vectors in $H_1(S_T,\Z)\subset H_1(S_T,\C)\cong\C^n$. More precisely, the $v_k$ are pairwise distinct, and there exist constants $a>0,b\geq 0$, such that
\[ \|v_k\|\leq ak+b,\quad \forall k\in\N. \]
Indeed, we found such families for every affine tree in the preceding section. Assume now that $\partial S_T$ is a knot. While the family $\{v_k\}$ could still decompose into finitely many orbits under the monodromy $\phi'_*$ of the smaller surface $S_{T'}$, we will show this cannot be the case for orbits of $\phi_*$, the monodromy of $S_T$. Namely, assume to the contrary that there were $r$ indices $k_1,\ldots,k_r\in\N$ such that the $\phi_*$-orbits of $v_{k_1},\ldots,v_{k_r}$ covered the whole sequence $\{v_k\}_{k\in\N}$. By Lemma~\ref{lem:jordanblocks} and Lemma~\ref{lem:exponential} (applied to $\Phi=\phi_*$), we obtain the following, for every $i\in\{1,\ldots,r\}$. Either the $\phi_*$-orbit of $v_{k_i}$ is bounded and thus finite, or there exist $h_i>1,c_i>0,d_i\geq 0$, such that $\|\phi_*^k(v_{k_i})\| \geq c_ih_i^k - d_i$ for all $k\in\N$. Replacing $h_i,c_i,d_i$ by minimal or maximal values $h>1,c>0,d\geq 0$ respectively, we have
\[ \|\phi_*^k(v_{k_i})\| \geq c h^k-d,\quad \forall k\in\N. \]
Now, let $K\in\N$ be large (to be specified later) and set $R:=aK+b$. We would like to compare the numbers
\begin{eqnarray}
   p&:=&\#\{k\in\N\ | \ \|v_k\|\leq R\},\nonumber \\
   q&:=&\#\{(k,i)\in\N\x\{1,\ldots,r\} \ | \ \|\phi_*^k(v_{k_i})\| \leq R\}. \nonumber
\end{eqnarray}
First, $p\geq K$, since $\|v_k\|\leq aK+b=R$ for all $k\leq K$, and the $v_k$ are pairwise distinct. Second, taking $N:=\left\lfloor\frac{K}{r}\right\rfloor -1$, we have $K>rN$, and
\[ ch^N-d=ch^{\lfloor\frac{K}{r}\rfloor -1}-d\geq aK+b=R, \]
for large enough $K$, since $h,c>1$. Then, we have $q\leq rN$, since $\|\phi_*^k(v_{k_i})\|\geq ch^N-d\geq R$ as soon as $k\geq N$. Thus $p\geq K>rN\geq q$, contradicting our assumption on the sequence $v_k$ being covered by the $\phi_*$-orbits of its members $v_{k_1},\ldots,v_{k_r}$. This finishes the proof.
\end{proof}

\begin{question}[by Pierre Dehornoy]
Can every embedded Hopf band in $S_T$ be obtained from one of the Hopf bands $v_i$ by successively applying the monodromies of $S_{T'}$, where $T'$ ranges over suitable subtrees of $T$?
\end{question}


\end{document}